\documentclass{article}
\usepackage[latin1]{inputenc}
\usepackage{amsmath}
\usepackage{amssymb}
\usepackage{amsthm}

\begin{document} 
 
\theoremstyle{plain} \newtheorem{Thm}{Theorem}
\newtheorem{prop}{Proposition}
\newtheorem{lem}{Lemma}
\theoremstyle{remark} \newtheorem*{pf}{Proof}
\renewcommand\theenumi{(\alph{enumi})}
\renewcommand\labelenumi{\theenumi}
\renewcommand{\qedsymbol}{}
\renewcommand{\qedsymbol}{\ensuremath{\blacksquare}}
\title{Elementary Abelian p-groups of rank 2p+3 \\ are not CI-groups}
\author{Gábor Somlai\\ 
Department of Algebra and Number Theory\\
E\"otv\"os University, Budapest, Hungary\\
email: zsomlei@gmail.com} 
\maketitle
\begin{abstract}

For every prime $p > 2$ we exhibit a Cayley graph of
$\mathbb{Z}_p^{2p+3}$  which is not a CI-graph. This proves that an
elementary Abelian $p$-group of rank greater than or equal to
$2p+3$ is not a CI-group. The proof is elementary and uses only
multivariate polynomials and basic tools of linear algebra. 
Moreover, we apply our technique to give a uniform explanation for the
recent works concerning the bound.

\end{abstract}

\section*{Introduction}

Let $G$ be a finite group and $S$ a subset of $G$. The Cayley graph
$Cay(G,S)$ is defined by having the vertex set $G$ and $g$ is adjacent
to $h$ if and only if $g h^{-1} \in S$. $S$ is called the connection
set of the Cayley graph $Cay(G,S)$. Every right translation is an
automorphism of $Cay(G,S)$, so the automorphism group of every
Cayley graph of $G$ contains a regular subgroup isomorphic to $G$
and this property characterises the Cayley graphs of $G$.

It is clear that $Cay(G,S) \cong Cay(G,S^{\sigma } )$ for every
$\sigma \in Aut(G)$ and these isomorphisms of the Cayley graphs are
called Cayley isomorphisms. A Cayley graph $Cay(G,S)$ is said to be a
CI-graph if for each $T \subset G$ the Cayley graphs $Cay(G,S)$ and
$Cay(G,T)$ are isomorphic if and only if there is an automorphism
$\sigma$ of $G$ such that $S^{\sigma } =T$. Furthermore, a group $G$
is called CI-group if every Cayley graph of $G$ is a CI-graph.

For our discussion two previous results are relevant.
It has been proved that if $G$ is a CI-group then the same holds
for every subgroup of $G$. Babai and Frankl proved in \cite{1} that
the Sylow subgroups of a CI-group can only be $\mathbb{Z}_4 $,
$\mathbb{Z}_8 $, $\mathbb{Z}_9$, $\mathbb{Z}_{27} $, the
quaternion group $Q$ or an elementary Abelian $p$-group and they asked
whether every elementary Abelian $p$-group is CI.
  
Hirasaka and Muzychuk proved in \cite{3} that $\mathbb{Z}_p^4 $ is a
CI-group for every prime $p$. On the other hand Muzychuk \cite{4}
proved that an elementary Abelian $p$-group of rank $2p-1+
\binom{2p-1}{p}$ is not a CI-group and most recently as a
strengthening of this result P. Spiga \cite{6} showed that if $n \ge
4p-2$ then $\mathbb{Z}_p^n$ is not a CI-group. The problem of
determining whether or not an elementary Abelian group
$\mathbb{Z}_p^n$ is CI is solved if $p=2$ as the CI property holds for
$\mathbb{Z}_2^5$, see \cite{2}, and a non-CI-graph for
$\mathbb{Z}_2^6$ was constructed by Nowitz \cite{5}. 

Further improving this bound we have the following.

\begin{Thm}\label{thm1}
For every prime $p >2$ the group $\mathbb{Z}_p^{2p+3}$ has a Cayley
graph of degree $(2p+3)p^{p+1}$ which is not a CI-graph.
\end{Thm}

The proof of the theorem is elementary and uses only the definition of
the CI property. We will construct two isomorphic Cayley graphs. The 
connection sets in both graphs are the union of affine hyperplanes in
$\mathbb{Z}_p^{2p+3}$  and the isomorphism between the Cayley graphs
is given in terms of polynomials. Finally, the proof that our Cayley
graphs are not CI graphs uses only elementary tools of linear algebra. 
In addition, we will indicate how the previous results of Muzychuk and
Spiga can be obtained applying our technique.

\section*{The construction}

Let $U \cong \mathbb{Z}_p^{p+1}$ and $V \cong
\mathbb{Z}_p^{p+2}$, then $U$ and $V$ can be regarded as vector spaces over
the field $\mathbb{Z}_p$ with bases $\{ e_1,e_2, \dotsc ,e_{p+1} \}$
  and $\{ f_0,f_1, \dotsc ,f_{p+1} \}$, respectively. We endow $V$ with
  the natural bilinear form given as follows: 
\[ \left< \sum_{i=0}^{p+1} \alpha_i f_i, \sum_{i=0}^{p+1}
  \beta_i f_i \right> = \sum_{i=0}^{p+1} \alpha_i \beta_i \mbox{.} \]
\\ Let us define the following affine subspaces in $U  \bigoplus V$ :
\begin{flalign*}
A_i &=e_i+ \left\{ \, v \in V \mid \left< v ,f_0
+f_i \right> =0 \, \right\} \mbox{,} && (i=1, \dotsc ,p+1) \\
B_i &=  \sum_{j \ne i} e_j + \left\{ \, v \in V \Bigg|
\left< v ,f_i + \sum_{j=0}^{p+1} f_j \right>  =0 \, \right\} \mbox{,}
&& (i=1, \dotsc , p+1) \\ C_0 &=
\sum_{j=1}^{p+1} e_j+ \left\{ \, v \in V \Bigg|
\left< v ,\sum_{j=0}^{p+1} f_j \right> =0 \, \right\} \mbox{,} \\ C_1 &=
\sum_{j=1}^{p+1} e_j+ \left\{ \, v \in V \Bigg|
\left< v ,\sum_{j=0}^{p+1} f_j \right>  =1 \, \right\} \mbox{.}
\end{flalign*}

Now 
\begin{equation} \label{set} S=\bigcup_{i=1}^{p+1} ( A_i \cup B_i )
  \cup C_0  \quad \mbox{ and } \quad T=\bigcup_{i=1}^{p+1} ( A_i \cup
  B_i )  \cup C_1
 \end{equation} 
will be the connection sets of two Cayley graphs defined on $G=U
\bigoplus V$. Note that $\left| S \right| = \left| T \right| = (2p+3)
p^{p+1}$ as desired.

We are going to show that $Cay(G,S) \cong Cay(G,S)$ but there is no
automorphism of $G$ mapping $S$ to $T$.

\section*{Preliminary facts}
In this section  we introduce some notation concerning polynomials and
we establish certain equations over the field $\mathbb{Z}_p$. These
will be used in the proof of the isomorphism between the two Cayley
graphs.

For a series of integers $\underline{n} := (n_1, \dotsc, n_{p+1})$
we denote $\underline{x}^{ \underline{n}}:=x_1^{n_1} \cdots
x_{p+1}^{n_{p+1}}$ and let $k \left( \underline{x}^{ \underline{n}}
\right) = \left| \left\{ \, i \mid n_i > 0 \, \right\} \right| $
denote the number of variables occuring in $
\underline{x}^{\underline{n}}$. Let $\mathcal{M}$ be the set of
monomials of degree $p$ involving at least two variables and for each
$i=1, \dotsc ,p+1$ we cut it into two subsets $\mathcal{M}=
\mathcal{M}_i^0 \cup \mathcal{M}_i^{+} $ , where
$\mathcal{M}_i^0=\left\{ \, \underline{x}^{\underline{n}} \mid  n_i=0
  \, \right \}$ and $\mathcal{M}_i^{+} = \left\{ \,
  \underline{x}^{\underline{n}} \mid n_i > 0  \, \right \}$. For a
monomial $\underline{x}^{\underline{n}} \in \mathcal{M}$ we define the
number $c_{\underline{n}} = \frac{(p-1)!}{n_1 ! \cdots n_{p+1} !}$. A
well known consequence of the Multinomial Theorem is that
$\frac{p!}{n_1 ! \cdots n_{p+1} !}$ is an integer. If
$\underline{x}^{\underline{n}} \in \mathcal{M}$, then $k \left(
  \underline{x}^{ \underline{n}} \right) \ge 2$ so $p$ does not divide
the denominator of $c_{\underline{n}}$ and hence
$c_{\underline{n}}$ is an integer.  Finally, for $\underline{\alpha }
\in \mathbb{Z}_p^k$ and $f(\underline{x} ) \in \mathbb{Z}_p [x_1,
\dotsc ,x_k ]$ we denote $$ \Delta_{\underline{\alpha } } f
(\underline{x} ) =f ( \underline{x} + \underline{\alpha} )-f(
\underline{x} ) \mbox{.} $$.

\begin{lem}\label{lem1}
Let $s= \sum_{i=1}^{p+1} x_i$ and $s_i=s-x_i= \sum_{j \ne i} x_j$.
\begin{enumerate}
\item \[ s^p= \sum_{j=1}^{p+1} {x_j}^p + \sum_{\underline{x}^{\underline{n}} \in
  \mathcal{M}} p c_n \underline{x}^{\underline{n}} \mbox{.} \]
\item \[ {s_i}^p= \sum_
{j \ne i} {x_j}^p + \sum_{\underline{x}^{\underline{n}} \in
  \mathcal{M}_{i}^{0}} p c_{\underline{n}}
\underline{x}^{\underline{n} } \mbox{.} \]
\end{enumerate}
\end{lem}
\begin{proof}
These identities are obvious.
\end{proof} 

Define the following polynomials in $\mathbb{Z}_p[x_1, \cdots ,
x_{p+1} ]$ :
\begin{align}
r_i&= \sum_{ \underline{x}^{\underline{n}} \in \mathcal{M}_i^{0}}
  \left( 1-k \left( \underline{x}^{\underline{n}} \right) \right) c_{\underline{n}}
  \underline{x}^{\underline{n}} + \sum_{ \underline{x}^{\underline{n}}
    \in \mathcal{M}_i^{+}} (2-k( \underline{x}^{\underline{n}} ))
  c_{\underline{n}}  \underline{x}^{ \underline{n}} \label{dd1}
 \intertext{for $i=1, \dots ,p+1$ and}
 r_{0}   &= \sum_{ \underline{x}^{\underline{n}} \in \mathcal{M}} (k(
\underline{x}^{ \underline{n}} )-2) c_{\underline{n}}  \underline{x}^{
  \underline{n}} \mbox{.} \label{dd2} 
\end{align}

\begin{lem}\label{lem2}
\[ \sum_{j=0}^{p+1} r_j=\frac{ps^p - \sum_{j=1}^{p+1} s_j^p}{p}
\mbox{.} \]
\end{lem}
\begin{proof}

\begin{align*}
\sum_{j=0}^{p+1} r_j &=\sum_{ \underline{x}^{
    \underline{n}} \in \mathcal{M}} \bigg( (p+1-k( \underline{x}^{
  \underline{n}} ))(1-k( \underline{x}^{ \underline{n}})) +  (k(
\underline{x}^{ \underline{n}})-1)(2-k( \underline{x}^{
  \underline{n}}))\bigg) c_{\underline{n}}  \underline{x}^{
  \underline{n}} \\ &=  (1-p)\sum_{ \underline{x}^{\underline{n}}
    \in \mathcal{M} } (k(  \underline{x}^{\underline{n}}) -1)
  c_{\underline{n} }  \underline{x}^ {\underline{n}} = \sum_{ \underline{x}^{\underline{n}}
    \in \mathcal{M} } ( k ( \underline{x}^{\underline{n}} ) -1)
  c_{\underline{n} }  \underline{x}^ {\underline{n}} 
\end{align*}
{and Lemma \ref{lem1} gives}
\[ 
\frac{ps^p - \sum_{j=1}^{p+1} s_j^p}{p} = \sum_{ \underline{x}^{\underline{n}}
    \in \mathcal{M} } ( k ( \underline{x}^{\underline{n}} ) -1)
  c_{\underline{n} }  \underline{x}^ {\underline{n}}
\]
{as well.}
\end{proof}

\section*{Isomorphism}

\begin{prop}
$Cay(G,S) \cong Cay(G,T)$.
\end{prop}

\begin{proof}
Let $\phi : \mathbb{Z}_p^{2p+3} \rightarrow \mathbb{Z}_p^{2p+3}$ be
defined by
\begin{multline*}
\phi \left( x_1,\dotsc , x_{p+1} , y_0, y_1, \dotsc , y_{p+1} \right) = \\ 
= \big( x_1, \dotsc , x_{p+1},y_0 + r_0 (x_1, \dotsc , x_{p+1}) ,  \dotsc
, y_{p+1} + r_{p+1} (x_1,  \dotsc , x_{p+1}) \big) 
\end{multline*}
 where $r_i \in \mathbb{Z}_p[x_1, \dotsc , x_{p+1} ]$ are defined in
 (\ref{dd1}) and (\ref{dd2}) .

We claim that $\phi$ is an isomorphism from $Cay(G,S)$ to
$Cay(G,T)$.
Note that $\phi$ acts by translation on $\underline{u}+V$ for
every $\underline{u} \in U$ so $\phi$ is bijective.
It remains to show that for $a \mbox{,} b \in G$ if $b-a \in S$ then $\phi (b)- \phi (a) \in
T$. 

Assume first that $b-a \in A_i$ for some $1 \le i \le p+1$ and write $a=(\underline{x} ,
\underline{y} )$ with $\underline{x} \in U$ and $\underline{y} \in V$.
 Clearly $\phi$ does not affect the first $p+1$ coordinates hence we
 need to show $\phi (b)- \phi (a) \in A_i$. 
Now we have
\begin{align*} 
\big( \phi (b)- \phi (a) \big) -\big( b-a\big) &=\big( \phi (b)-b
\big) - \big(
\phi (a)-a \big) = \\ &=(0, \dotsc ,0,\Delta_{e_i}
r_0(\underline{x}),\Delta_{e_i} r_1 (\underline{x}), \dotsc
,\Delta_{e_i} r_{p+1} (\underline{x})) \mbox{.} 
\end{align*}
\\Thus we have to check that $ \left< ( \Delta_{e_i}
r_0(\underline{x} ),\Delta_{e_i} r_1 (\underline{x}), \dotsc
,\Delta_{e_i} r_{p+1} (\underline{x})),f_0+f_i \right> =0$.
Now
\begin{align*}
\left< ( \Delta_{e_i}
r_0(\underline{x}),\Delta_{e_i} r_1 (\underline{x}), \dotsc
,\Delta_{e_i} r_{p+1} (\underline{x}) ) ,f_0+f_i \right> &=\Delta_{e_i} r_0
(\underline{x})+\Delta_{e_i} r_i (\underline{x} ) = \\ &=\Delta_{e_i}
(r_0 (\underline{x}) +r_i (\underline{x}))=0
\end{align*}       
since $r_0 +r_i$ does not involve $x_i$.

By the same argument if $b-a \in C_0$ then using Lemma \ref{lem2} we
get 
\begin{align*}
&\Delta_{\sum_{j=1}^{p+1} e_j } \left( \sum_{j=0}^{p+1} r_j \right)
=\frac{p(s+p+1)^p -
\sum_{j=1}^{p+1} (s_j+p)^p}{p} - \frac{ps^p - \sum_{j=1}^{p+1}
s_j^p}{p} \\ &= (s+1)^p-s^p=1 \mbox{.}
\end{align*}
These equations hold over $\mathbb{Z}_p$ since $(t +p)^p \equiv
t^p \pmod{p^2}$. Hence if $b-a \in C_0$, then $\phi (b) - \phi (a) \in
C_1$.

Finally, if $b-a \in B_i$ we need a little more computation.
The proof of Lemma \ref{lem2} shows that 
\[ \sum_{j=0}^{p+1} r_j= \sum_{ \underline{x}^{\underline{n}}
    \in \mathcal{M} } ( k ( \underline{x}^{\underline{n}} ) -1)
  c_{\underline{n} }  \underline{x}^ {\underline{n}} \mbox{.} \]
Hence
\begin{multline*}
r_i + \sum_{j=0}^{p+1} r_j =\sum_{ \underline{x}^{\underline{n}} \in \mathcal{M}_i^{0}}
  (1-k( \underline{x}^{\underline{n}} )) c_{\underline{n}}
  \underline{x}^{\underline{n}} + \sum_{ \underline{x}^{\underline{n}}
    \in \mathcal{M}_i^{+} } (2-k( \underline{x}^{\underline{n}} )) c_{\underline{n}}
  \underline{x}^{\underline{n}} \\
+  \sum_{ \underline{x}^{\underline{n}}
\in \mathcal{M} } ( k(  \underline{x}^{\underline{n}} ) -1 )
  c_{\underline{n} }  \underline{x}^ {\underline{n}}  =  \sum_{
    \underline{x}^{\underline{n}} \in \mathcal{M}_i^{+}} c_{\underline{n} }
  \underline{x}^{\underline{n} } = \frac{s^p-x_i^p-s_i^p}{p} \mbox{.}
\end{multline*}
Therefore
\[ \Delta_{\sum_{j \ne i} e_j} \left( r_i + \sum_{j=0}^{p+1} r_j
\right) =
\frac{(s+p)^p- x_i^p - (s_i+p)^p}{p} -\frac{s^p- x_i^p - s_i^p}{p} =0 \]
using again the fact that $(t+p)^p \equiv t^p \pmod{p^2}$.
Hence if $b-a \in B_i$, then $\phi (b) - \phi (a) \in B_i$ and this
finishes the proof of the fact that $\phi$ is indeed a graph isomorphism.
\end{proof}

\section*{Checking the CI property}

Now in order to show that $Cay(G,S)$ is not a CI graph we have to show
that there is no $\sigma \in Aut(G)=GL(U \bigoplus V)$ such that
$\sigma (S) =T$.
\begin{prop}\label{prop2}
There is no linear transformation $\sigma \in GL(U \bigoplus V)$ such
that $\sigma (S)=T$.
\end{prop}
\begin{proof}
Assume by way of contradiction that $\sigma \in GL(U \bigoplus V )$ with
$\sigma (S)=T$. Let $M$  denote the matrix of the linear
transformation $\sigma$ with respect to the  the basis
$\{ e_1, \ldots ,e_{p+1}, f_0, f_1, \ldots ,f_{p+1} \}$ and  write
$M=  \left[ \begin{array}{cc} M_{1,1} & M_{1,2} \\
  M_{2,1} & M_{2,2} \end{array} \right] $ as a block matrix, where
$M_{1,1} \in \mathbb{Z}_p^{(p+1) \times (p+1)} $ and $M_{2,2} \in
\mathbb{Z}_p^{(p+2) \times (p+2)} $ \text{.} 

For the purpose of the following we modify our notation as
follows. Let $S= \mathop{\cup}\limits_{i=1}^{2p+3} S_i$
and $T= \mathop{\cup}\limits_{i=1}^{2p+3} T_i$, where $S_i =T_i =A_i$,
$S_{i+p+1} =T_{i+p+1} =B_i$ for $i=1, \ldots ,p+1$ and $S_{2p+3}=C_0$ ,
$T_{2p+3}=C_1$. 
\begin{lem}\label{lem3}
V is an invariant subspace of $\sigma$, i.e., $M_{1,2}=0$.
\end{lem}

\begin{proof}
Considering only the first $p+1$ coordinates it is easy to see that for
$i \ne j$ if $a \in S_i$ and $b \in S_j$, then $2a-b
\not\in S $ and similarly for $T$ and thus both $S$ and $T$ contain
exactly $2p+3$ affine subspaces of dimension $p+1$. Hence for $1 \le i
\le 2p+3$ we must have $\sigma (S_i) \subset T_j$ for some $j$ and if $a$, $b$
$\in S_i$ then $\sigma (a)- \sigma (b) \in V$. Now $$Span \bigg( \,
\mathop{\cup}\limits_{i=1}^{p+1} \{ a-b \mid  a, b
  \in S_i \} \, \bigg)  =V$$ so $\sigma(V) \subseteq V$. 
\end{proof}
It is immediate from the preceding lemma that $\sigma$ is a linear
transformation of $(U \bigoplus V)/V$ and for $\hat{S} = \left\{ \,
  e_i, \sum_{j \ne i} e_j \mid 1 \le i \le p+1 \, \right\} \cup
\left\{ \, \sum_{j=1}^{p+1} e_j \, \right\} \subset U$ we have
$\sigma(\hat{S})= \hat{S}$.

\begin{lem}\label{lem4}
$M_{1,1}$ is a permutation matrix.
\end{lem}

\begin{proof}
Let $e:=\sum_{j = i}^{p+1} e_j$.  Note that $e$ is the unique element
of $\hat{S}$ which is the sum of two others within $\hat{S} $ hence $\sigma (e) =
e$. The rest of the points can be paired such that the sum of every
pair is $e$ and by the linearity of $\sigma $ the set $H= \{ \,\sigma (e_i)
\mid 1 \le i \le p+1 \, \}$ contains exactly one element of each
pair. Furthermore, $\sum_{h \in H} h = \sum_{i=1}^{p+1} \sigma (e_i) = \sigma
(\sum_{i=1}^{p+1} e_i ) = \sigma (e)=e $. 

For every $s \in \hat{S}$ $ \left< s,e  \right> = 0 \mbox{
  or  } 1$ hence if $H$ contains an element $x$ such that $\left< x,e
\right> =0$ then $H$ contains $p$ elements with the same property as
$\left< \sum_{h \in H} h, e \right> = \left< e,e \right> =1$. By
permuting the coordinates we obtain that if H contains an element $h$
such that $\left< h, e \right> =0$ then $H=\{ \, e_1 \, \} \cup \{ \,
\sum_{j \ne i } e_j \mid 2 \le i \le p+1 \, \}$ but then $\sum_{h \in
  H} h = e_1 -e_2 - \dotsc - e_{p+1} \ne \sum_{j=1}^{p+1} e_j =e $.
\end{proof}

For every permutation of $\{ e_1, \dotsc ,e_{p+1} \}$ if we apply
the same permutation to the indices of $\{ f_1, \dotsc ,f_{p+1} \}$ and fix $f_0$
we obtain an automorphism of $Cay(G,S)$. Hence we may assume for the
rest of the proof that $M_{1,1}=I$. This assumption implies that
$\sigma (e_i ) \in A_i$ and $\sigma (\sum_{j \ne i} e_j ) \in B_i$ for
$1 \le i \le p+1$.
From this we get 
\begin{align*}
\left< M_{2,1} e_i ,f_0 +f_i \right> &=0  \mbox{,}  \\ \left<
  M_{2,1} \sum_{j \ne i} e_j ,f_i + \sum_{j=0}^{p+1} f_j \right>
&=0 
\end{align*}
for $1 \le i \le p+1$.
\\ The sum of these equations over $\mathbb{Z}_p$ is \[ \left< M_{2,1}
\sum_{i=1}^{p+1} e_i , \sum_{j=0}^{p+1} f_j \right> =0 \mbox{ .}\]
We also have that $\sigma ( \sum_{j=1}^{p+1} e_j ) \in C_1$ which
gives the following condition: 
\[ \left< M_{2,1} \sum_{i=1}^{p+1} e_i , \sum_{j=0}^{p+1} f_j \right>
=1 \]
and this contradicts what we have just found above, finishing the
proof of Proposition \ref{prop2}.
\end{proof}

\section*{Undirected graphs}

It is clear that a Cayley graph $Cay(G,S)$ is undirected if and only
if $S=S^{-1}$, where $S^{-1} = \left\{ \, s^{-1} \in G \mid s \in S \,
  \right\} $. If $G$ is  an Abelian group we write $-S=\left\{ \, -s
    \in G \mid s \in G \, \right\}$ instead of $S^{-1}$. For a subset $S$ of $G$ we define
  $\bar{S} = S \cup -S$. It is also clear that if $\phi$ is an
  isomorphism between $Cay(G,S)$ and $Cay(G,T)$ then  $\phi$ is an
  isomorphism between $Cay(G,\bar{S} )$ and $Cay(G,\bar{T} )$ as well.

In the previous sections we found two isomorphic directed Cayley graphs
$Cay(\mathbb{Z}_p^{2p+3}, S)$ and $Cay(\mathbb{Z}_p^{2p+3}, T)$ of
$\mathbb{Z}_p^{2p+3}$, where $S$ and $T$ are defined in (\ref{set}).
Therefore we have a pair of isomorphic undirected Cayley graphs:
$Cay(\mathbb{Z}_p^{2p+3}, \bar{S} )$ and $Cay(\mathbb{Z}_p^{2p+3},
\bar{T} )$.

\begin{Thm}\label{thm2}
For every prime  $p>3$  the group $\mathbb{Z}_p^{2p+3}$ has an
undirected Cayley graph which is not a CI graph.
\end{Thm}

\begin{pf}
It is enough to show that there is no linear transformation $\sigma$
such that $\sigma (\bar{S}) = \bar{T}$.  Let us assume that $\sigma
\in GL(U \bigoplus V )$ with $\sigma (\bar{S} ) = \bar{T}$. 

The same kind of reasoning as in the previous section shows that $V$
is an invariant subspace of $\sigma$, but here we have to use the
extra condition that $p > 3$. We also get a subset \[ \tilde{S}=
\left\{ \, e_i, -e_i, \sum_{j \ne i} e_j, - \sum_{j \ne i} e_j \bigg|
  1 \le i \le p+1 \, \right\} \cup \left\{ \, \sum_{j=1}^{p+1} e_j,
  -\sum_{j=1}^{p+1} e_j \, \right\} \] of $U$ such that $\sigma
(\tilde{S} ) =\tilde{S}$. Note that we can write $\tilde{S} = \hat{S} \cup
-\hat{S}$ with $\hat{S} \cap -\hat{S} =\emptyset $.
\end{pf}

\begin{lem}
One of the two linear transformations $\sigma$ and $- \sigma$
permutes the elements of $\hat{S}$.
\end{lem}

\begin{pf}
Since $\sigma$ induces an automorphism of $Cay(U,\tilde{S})$ and
$\sigma(0)=0$, it gives an automorphism of the induced subgraph on
the neighborhood of $0$ as well. In this subgraph the vertices $e$ and
$-e$ have degree $2p+2$, the other vertices have degree 2. This
implies that $\sigma(e)=e$ or $\sigma(e)=-e$. So either $\sigma$ or
$-\sigma$ fixes $e$. The neighborhood of $e$ in $\tilde{S}$ is
$\hat{S}$, hence the proof of Lemma~\ref{lem4} yields the result. \qed
\end{pf}

As a consequence of the previous lemma we get a linear transformation
($\sigma$ or $-\sigma$) which maps $S$ onto $T$ but this contradicts
Proposition \ref{prop2}, finishing the proof of Theorem
\ref{thm2}. \qed


\section*{Connection to previous results}
In this section we modify our construction a little bit to get
a non CI graph of the groups $ \mathbb{Z}_p^{4p-2}$ and  $
\mathbb{Z}_p^{2p-1+ \binom{2p-1}{p}}$. These results provide a uniform
explanation for the recent work of P. Spiga \cite{6} and M. Muzychuk
\cite{4}, respectively.

\subsection*{Rank $4p-2$ }
Let $U' \cong V' \cong \mathbb{Z}_p^{2p-1}$ and $W'=U' \bigoplus V'$
with the bases $\{ e_1', \dotsc ,e_{2p-1}' \}$ and $\{ f_1', \dotsc
,f_{2p-1}' \}$, respectively. We denote by $\mathcal{L}$ the set of
multilinear monomials of degree $p$ in $2p-1$ variables. Let
$\mathcal{L}_i^{0} =\left\{ \, \underline{x}^{\underline{n}} \in
  \mathcal{L} \mid n_i=0 \, \right\}$ and $\mathcal{L}_i^{+} =
\mathcal{L} \setminus \mathcal{L}_i^{0} $. If
$\underline{x}^{\underline{n}} \in \mathcal{L}$ then the exponent
vector $\underline{n}$
can be treated as a $p$-element subset of $\{ \,1, \ldots , 2p-1 \, \}$.
\\Let
\begin{align*}
A_i' &=e_i' + \left\{ \, v' \in V' \mid \left< v' ,f_i'
\right> =0 \, \right\} \mbox{,}  \\
B_i' &=  \sum_{j \ne i} e_j' + \left\{ \, v' \in V' \Bigg|
\left< v' ,f_i' + \sum_{j=1}^{2p-1} f_j' \right>  =0 \, \right\}
\mbox{,} \\ C_0' &=
\sum_{j=1}^{2p-1} e_j' + \left\{ \, v' \in V' \Bigg|
\left< v' ,\sum_{j=1}^{2p-1} f_j' \right> =0 \, \right\}
\mbox{,}  \\ C_1' &=
\sum_{j=1}^{2p-1} e_j' + \left\{ \, v' \in V' \Bigg|
\left< v' ,\sum_{j=1}^{2p-1} f_j' \right>  = -1 \, \right\}
\mbox{.}
\end{align*}
Similarly to the original construction let $S'=\bigcup_{i=1}^{2p-1} ( A_i'
\cup B_i' )  \cup C_0'$ and $T'=\bigcup_{i=1}^{2p-1} (
 A_i' \cup B_i' )
\cup C_1'$.  We claim that $Cay(W',S') \cong Cay(W',T')$ and the
isomorphism is given in the same manner:
\begin{multline*}
\phi ' (x_1,\dotsc , x_{2p-1} , y_1, \dotsc ,y_{2p-1})= \\ 
= \big( x_1, \dotsc , x_{2p-1} , y_1 + l_1 (x_1, \dotsc , x_{2p-1}) ,
\dotsc , y_{2p-1} + l_{2p-1} (x_1,  \ldots , x_{2p-1}) \big) 
\end{multline*}
where $l_i$ denotes the sum of the monomials in $\mathcal{L}_i^{0} $
for $i=1, \dotsc ,2p-1$.
\\In this case the computations needed to show that $\phi '$ is an
isomorphism of the two Cayley graphs are easier.
\begin{lem}\label{lem5}
Assume that $\underline{x}^{\underline{n}} \in \mathcal{L} $ and
$\underline{m} \in \{\, 0,1 \, \}^{2p-1} \subseteq U'$

\begin{enumerate}
\item \label{a1} 
\[ ( \Delta_{\underline{m}} \underline{x}^{\underline{n}} ) (\underline{x})
= \underline{x}^{\underline{n} \setminus \underline{m} } \sum_{\underline{k} \mspace{2mu} \subsetneqq \mspace{2mu}
  \underline{n} \cap \underline{m} } \underline{x}^{\underline{k}}
\mbox{.} \]
\item \label{a2}
\[ (\Delta_{\underline{1}} \underline{x}^{\underline{n}} )(\underline{x})
= \sum_{\underline{k} \mspace{2mu} \subsetneqq \mspace{2mu} \underline{n}}
 \underline{x}^{\underline{k} } \mbox{.} \]

\end{enumerate}

\end{lem}
\begin{proof}
\ref{a1} is obvious and \ref{a2} is just a particular case of \ref{a1}.
\end{proof}

Following the way of the proof it turns out that it suffices to verify
the following three equations.
\begin{lem} $\mspace{50mu}$
\begin{enumerate}
\item \label{i1}
\[ \Delta_{e_i'} l_i =0 \mbox{.} \]
\item \label{i2}
\[ \Delta_{\sum_{j=1}^{2p-1} e_j' } ( \sum_{j=1}^{2p-1} l_j ) =-1
\mbox{.} \]
\item \label{i3}
\[ \Delta_{\sum_{j \ne i} e_j'} (l_i + \sum_{j=0}^{2p-1} l_j) =0
\mbox{.} \] 

\end{enumerate}
\end{lem}

\begin{proof} \hspace{1cm} \\
\ref{i1}
$l_i$ does not involve $x_i$. \\
\ref{i2}
\begin{equation} \label{e1}
\sum_{j=1}^{2p-1} l_j = \sum_{j=1}^{2p-1}
\sum_{\underline{x}^{\underline{n}} \in \mathcal{L}_i^0 }
\underline{x}^{\underline{n}} = \sum_{\underline{x}^{\underline{n}}
  \in \mathcal{L} } (p-1) \underline{x}^{\underline{n} } = -
\sum_{\underline{x}^{\underline{n}} \in \mathcal{L} }
\underline{x}^{\underline{n} } 
\end{equation}
and hence
\begin{align*}
&\Delta_{\sum_{j=1}^{2p-1} e_j' } ( \sum_{j=1}^{2p-1} l_j ) = -
\Delta_{\sum_{j=1}^{2p-1} e_j'} \sum_{\underline{x}^{\underline{n}}
  \in \mathcal{L} } \underline{x}^{\underline{n} } = - \sum_{\underline{x}^{\underline{n}}
  \in \mathcal{L} } \Delta_{\sum_{j=1}^{2p-1} e_j' }
\underline{x}^{\underline{n} } \\
\intertext{applying Lemma \ref{lem5} \ref{a2}}
&= - \sum_{\begin{subarray}{l}  \underline{n} \in \{ \, 0,1 \, \}^{2p-1} \\
    \left| \underline{n} \right| =p \end{subarray} } \mspace{6mu}
\sum_{\underline{k} \mspace{2mu} \subsetneqq \mspace{2mu}
  \underline{n}}  \underline{x}^{\underline{k}} = - \sum_{ \left|
  \underline{k} \right|  < p  } \underline{x}^{\underline{k}}
\sum_{\begin{subarray}{l} \underline{k} \mspace{2mu} \subseteq
    \underline{n}  \\ \left|
    \underline{n} \right| =p  \end{subarray} } 1 \\ &=- \sum_{ \left|
  \underline{k} \right| < p  } \binom{2p-1- \left| \underline{k} \right|
}{p- \left| \underline{k} \right| } \underline{x}^{\underline{k}} \mbox{.}
\end{align*}
The binomial coefficient $\binom{2p-1- \left| \underline{k} \right| 
}{p- \left| \underline{k} \right| } $ is divisible by $p$ if $ 1 \le \left|
\underline{k} \right| <p$ and this implies that  the remaining
polynomial is just the constant polynomial $- \binom{2p-1}{p} $  over
$\mathbb{Z}_p$. Taking into account that $\binom{2p-1}{p} \equiv 1
\pmod{p}$, we obtain \ref{i2}.
\\
\ref{i3} 
Substituting (\ref{e1}) into the equation we get
\[ l_i + \sum_{j=1}^{2p-1} l_j = \sum_{\underline{x}^{\underline{n}}
  \in \mathcal{L}_i^0 } \underline{x}^{\underline{n}} -
\sum_{\underline{x}^{\underline{n}} \in \mathcal{L} }
\underline{x}^{\underline{n} } = - \sum_{\underline{x}^{\underline{n}}
    \in \mathcal{L}_i^+ } \underline{x}^{\underline{n}} \mbox{ .} \]
Now
\begin{align*}
&\Delta_{\sum_{j \ne i} e_j'}
  \left( - \sum_{\underline{x}^{\underline{n}} \in \mathcal{L}_i^+ }
  \underline{x}^{\underline{n}} \right) = -
\sum_{\underline{x}^{\underline{n}} \in \mathcal{L}_i^+ }
\Delta_{\sum_{j \ne i} e_j'} \underline{x}^{\underline{n}} \\
\intertext{by Lemma \ref{lem5} \ref{a1}} &=-
\sum_{\underline{x}^{\underline{n} } \in \mathcal{L}_i^+ } x_i
\sum_{\underline{k} \subsetneqq \underline{n} \setminus \{ i \} }
\underline{x}^{\underline{k} } 
= - x_i  \sum_{ \begin{subarray}{l} i \notin \underline{k} \\ \left|
      \underline{k} \right|  < p-1 \end{subarray} }
\underline{x}^{\underline{k} } \sum_{ \begin{subarray}{l} \{ i \} \cup
    \underline{k} \subsetneqq \underline{n} \\ \left| \underline{n} \right|
    =p \end{subarray} } 1 \\ &= - x_i \sum_{ \begin{subarray}{l} i
  \notin \underline{k} \\ \left| \underline{k} \right|  < p-1
\end{subarray} } \binom{2p-1- \left| \underline{k} \right|  -1 }{p- \left|
\underline{k} \right|  -1 } \underline{x}^{\underline{k}} \mbox{.}
\end{align*}
Now if $\left| k \right| < p-1$ then $\binom{2p-1- \left|
    \underline{k} \right| -1 }{p- \left| \underline{k} \right|  -1 }
\equiv 0 \pmod{p}$ and this proves the result.
\end{proof}

The proof of the fact that there is no linear transformation which
maps $S'$ to $T'$ is nearly the same as in the previous case provided
to $p >3$. If $p=3$ then the statement analogous to Lemma \ref{lem4}
does not hold. We leave it to the reader to work out the details and
we will do so in the next case as well.

\subsection*{Rank $2p-1+ \binom{2p-1}{p}$ } 
Here we only give the connection sets and the isomorphism of the 
Cayley graphs. The proof goes along the same lines as in the previous
cases.

Let $\mathcal{O}= \{ \, \underline{k} \subset \{ \, 1, \ldots , 2p-1
\, \} \bigm| \mspace{1mu} \left| \underline{k} \right| =p \,
\}$ and let $U'' \cong \mathbb{Z}_p^{2p-1}$ and $V'' \cong
\mathbb{Z}_p^{\binom{2p-1}{p}}$ with the bases $\{ \, e_1'',e_2'', \dotsc
,e_{2p-1}'' \, \}$ and $\{ \, f_{ \underline{k} }'' \mid \underline{k}
\in \mathcal{O} \, \}$, respectively.
For  $\left( \underline{x''} , \underline{y''} \right) \in U'' \bigoplus V''$
we define
\[ \phi'' (\underline{x}'' , \underline{y}'')= \left( \underline{x}'' ,
  \dotsc , y_{\underline{k} }'' + {\underline{x}''}^{\underline{k}} ,
  \dotsc  \right) \mbox{.} \]
For a each $1 \le i \le 2p-1 $ we define the set $$A_i= e_i''+ \left\{ \, v'' \in
V'' \Bigg| \left< v'' , \sum_{i \notin \underline{k} } f_{\underline{k}}'' \right> =0 \,
  \right\} \mbox{.} $$ For every $\underline{k} \in \mathcal{O}$ there are exactly $p$
  elements $\underline{k}_1, \dotsc , \underline{k}_p$ of $\mathcal{O}$ such that
$\left| \underline{k} \cap \underline{k}_i \right|=1$ and hence we can
define $$B_{\underline{k}} = \sum_{j \in k} e_j'' + \{ \, v'' \in V''
\mid  \left< v'' , f_{\underline{k}_1}'' + \ldots +
  f_{\underline{k}_p}'' \right> =0 \, \} \mbox{.} $$ The third type of
affine subspaces are defined by $$C_0'' = \sum_{j=1}^{2p-1} e_j'' + \{
\, \underline{v}'' \in V'' \mid \left< \underline{v}''
  ,\sum_{\underline{k} \in \mathcal{O}} f_{\underline{k}}'' \right> =0
\, \}$$ and $$C_1'' = \sum_{j=1}^{2p-1} e_j'' + \{ \, \underline{v}''
\in V'' \mid \left< \underline{v}'' ,\sum_{\underline{k} \in
    \mathcal{O}} f_{\underline{k}}'' \right> =1 \, \}\mbox{.}$$
Finally, the connection sets are given similarly to the previous cases: \[S'' = \bigg(
  \mathop{\bigcup}\limits_{i} A_i \bigg)
\cup \bigg( \mathop{\bigcup}\limits_{\underline{k} \in \mathcal{O} } B_{\underline{k}} \bigg)
\cup C_0'' \] and \[ T'' = \bigg( \mathop{\bigcup}\limits_{i} A_i
\bigg) \cup \bigg( \mathop{\bigcup}\limits_{\underline{k} \in \mathcal{O} } B_{\underline{k}}
\bigg) \cup C_1'' \] 
and $\phi''$ gives the isomorphism between the two Cayley graphs.


\begin{thebibliography}{9}
\bibitem{1} L. Babai, P. Frankl, Isomorphism of Cayley graphs I,
      Colloquia Mathematica Societatis János Bolyai, Vol 18.
      Combinatorics, Keszthely (1976), North-Holland, Amsterdam, 35-52
      (1978)
\bibitem{2} M. Conder, C. H. Li, On isomorphism of Cayley graphs,
      European J. Combin. 19, 911-919 (1998)
\bibitem{3} M. Hirasaka, M. Muzychuk, An elementary Abelian group of
      rank 4 is a CI-group, J. Comb. Theory Ser. A 94(2), 339-362 (2001)
\bibitem{4} M. Muzychuk, An elementary Abelian group of large rank is
      not a CI-group, Discrete Math. 264(1-3), 167-185 (2003)
\bibitem{5} L. A. Nowitz, A non-Cayley-invariant Cayley graph of the
elementary Abelian group of order 64, Discrete Math. 110, 223-228
(1992)
\bibitem{6} P. Spiga, Elementary Abelian p-groups of rank greater than
      or equal to 4p-2 are not CI-groups, J. Algebr. Comb. 26, 343-355 (2007)
\end{thebibliography}
\end{document}